\documentclass[11pt]{amsart}
\usepackage[shortalphabetic]{amsrefs}
\usepackage[utf8]{inputenc}
\usepackage{amsthm}
\usepackage[utf8]{inputenc}
\usepackage{appendix}
\usepackage{cite}
\usepackage{amssymb,amsmath,amsthm}
\usepackage{hyperref}

\newcommand{\amsprimary}[1]{{\footnotesize\noindent AMS 2010 \textit{Mathematics subject
classification:} Primary #1\vspace{1pc}}}
\newcommand{\keywordsnames}[1]{{\footnotesize\noindent\textit{Key words:} #1\vspace{1pc}}}

\newtheorem{theorem}{Theorem}

\newtheorem{teo}{Theorem}

\newtheorem{lemma}[teo]{Lemma}

\theoremstyle{definition}

\theoremstyle{remark}

\title[Dirichlet problem manifolds with multiple ends]{On the Dirichlet problem at infinity on three-manifolds 
with multiple ends}
\author{Jean C. Cortissoz \and Ram\'on Urquijo}
\email{jcortiss@uniandes.edu.co, ra.urquijo66@uniandes.edu.co}
\address{Department of Mathematics, Universidad de los Andes, Bogot\'a DC, Colombia}
\date{}

\begin{document}

\maketitle

\begin{abstract}
In this paper we prove that in a three-manifold with finitely many  
expansive ends, such that
each end has a neighborhood where the curvature is bounded above by a negative constant,
the Dirichlet problem at infinity is solvable, and hence that
such manifolds posses a wealth of bounded non constant
harmonic functions (and thus, Liouville's theorem does not hold). In the case of infinitely many expansive ends,
we show that if each end 
has a neighborhood where the curvature is bounded above by a negative constant,
then the Dirichlet problem at infinity is solvable 
for continuous boundary data at infinity which is uniformly bounded. Our method is based on
a result that does not
require explicit curvature assumptions, and hence
it can be applied to other
situations: we present
an example of a metric 
on an end with curvature
of indefinite sign (no matter
how long we go along the end) for which the Dirichlet Problem at Infinity is solvable with respect to that end.
We also present a related result in the case of surfaces with a pole which generalises
a celebrated criteria of Milnor.
\end{abstract}

\keywordsnames{Dirichlet problem at infinity, bounded harmonic function, multiple ends.}

{\amsprimary {31C05, 53C21}}

\tableofcontents

\section{Introduction}

\subsection{the three-dimensional
case} An $n$-dimensional open manifold with cylindrical ends is an $n$-manifold $M$ with a relatively compact
open submanifold $U\subset M$, such that $\partial \overline{U}=E$, where 
$M\setminus U=E\times\left[0,\infty\right)$ and
\[
E=\bigsqcup_{j} E_j,
\]
where each $E_j$ is a smooth closed $(n-1)$-submanifold of $M$. 
Each of the $E_j$ is called a neighborhood of an end (here $\bigsqcup$ represents disjoint union).

\medskip
An expansive end of $M$ is an end which has a neighborhood diffeomorphic to $N\times \left[0,\infty\right)$,
such that the metric in that neighborhood can be written as
\[
g=dr^2+\phi\left(\omega,r\right)^2g_N,
\]
$\phi$ a smooth function, and such that there is an $R>0$ so that if $r\geq R$ then $\phi_r>0$, and
$\phi\rightarrow \infty$ as $r\rightarrow\infty$. 

\medskip
To the careful reader, the way 
we need a metric to be written in the neighborhood of an end might seem a bit restrictive at 
first sight. However, to get a better appreciation of its significance, it must
be also observed that all that is needed in order to be able 
to write $g$ on the neighborhood of an end as described above,
is that the end can be foliated by the level surfaces of the distance function
from $N\times\left\{0\right\}$, and that each of the
leaves of the foliation 
is a smooth submanifold diffeomorphic to $N$ (which 
actually follows if the distance function from 
$N\times\left\{0\right\}$ is smooth); then, by the Uniformization Theorem, the metric on each
of these level surfaces can be written as $\phi\left(\omega, r\right)^2g_N$ for a fixed 
metric on $N$, and hence the metric $g$ on the neighborhood of the end has the form thus described.

\medskip
To compactify a neighborhood of an end, we just add a copy of $N$ at infinity, that is we first compactify 
$\left[0,\infty\right)$ to obtain $\left[0,\infty\right]$ and then we obtain the compactification
of the neighborhood of the end by taking
\[
\overline{\mbox{Neighborhood of end}}=N\times\left[0,\infty\right].
\]
All is now set to explain what we mean by the Dirichlet problem at 
infinity to be solvable with respect to an end. We say that \emph{the Dirichlet problem is 
solvable at infinity with respect to an end} with neighborhood
homeomorphic to $N\times\left[0,\infty\right)$, and boundary data
$f:N\longrightarrow \mathbb{R}$,
if there is a harmonic function $u:M\longrightarrow \mathbb{R}$ such that when restricted to the neighborhood
 of the end,
it can be extended continuously to the neighborhood's compactification (that is, 
$u\left(\omega,r\right)\rightarrow f\left(\omega,r\right)$ pointwise
as $r\rightarrow \infty$ 
in the topology of the compactification). Recall that a function 
$u:M\longrightarrow \mathbb{R}$ is harmonic if it is in $C^{2}\left(M\right)$
and $\Delta_g u=0$, where $\Delta_g$ is the Laplace-Beltrami operator defined 
with respect to the metric $g$.

\medskip
On the other hand, we say that the Dirichlet problem is solvable at infinity on a manifold with 
ends if given its ends $\left\{\varepsilon_j\right\}_{j=1,2,3,\dots}$, with each end $\varepsilon_j$
having a neighborhood homeomorphic to $N_j\times\left[0,\infty\right)$,
and with $f_j\in C\left(N_j\right)$, there is a harmonic function $u$ in $M$ such that 
$u$ solves the Dirichlet problem with respect to each of the ends.

\medskip
As might be hinted from its title, the purpose of this paper is to
study the Dirichlet problem at infinity on open three-manifolds
 with expansive ends. Indeed, we prove the following.

\begin{theorem}
\label{thm:main2}
Let $\left(M^3,g\right)$ a Riemannian manifold with finitely many cylindrical ends, all of which are expansive.
Assume that each end has a neighborhood such that there is a constant $a\neq 0$ so that
the curvature in the given neighborhood of the end satisfies $K\leq -a^2$. Then the Dirichlet problem is solvable at infinity.
\end{theorem}

Before we proceed, let us make a quick
remark on the hypotheses of the previous theorem. For 
a neighborhood of 
an expansive end $N\times \left[0,\infty\right)$ with
a metric as described 
above, the bound
on the curvature is only needed
for the sectional curvatures of
planes that contain a normal
and a tangential direction to
$N$ (and which in turn can be
computed, as we will show
below, as $-\phi_{rr}/\phi$).

\medskip
We also present the proof of a result in the case when the manifold has infinitely many expansive ends,
but before we present its statement, we need to give another definition.
We say that the boundary data $f_j:N_j\longrightarrow \mathbb{R}$ is 
\emph{uniformly bounded from below} if
there is an $m\in \mathbb{R}$ such that 
\[
f_j\geq m\quad \mbox{holds} \quad \mbox{for all}\quad j.
\]
In a similar fashion we define the boundary data being \emph{uniformly bounded from above},
and we shall say that the boundary data is \emph{uniformly bounded} if it is uniformly
bounded from both above and below.

\medskip
We are now ready to state our theorem regarding three-manifolds with infinitely many expanding ends.

\begin{theorem}
\label{thm:main3}
Let $\left(M^3,g\right)$ a Riemannian manifold with infinitely many 
cylidrical ends, all of which are expansive.
Assume that each end has a neighborhood such that there is a constant $a\neq 0$ so that
the curvature in the given neighborhood of the end satisfies $K\leq -a^2$. 
Then, if the boundary data is uniformly
bounded, the Dirichlet problem is solvable at infinity.
\end{theorem}

\medskip
Theorems \ref{thm:main2} and \ref{thm:main3} are closely related to Proposition 3.9 in \cite{Choi},
which in part states that on a three-manifold with a pole,
such that the metric can be written as
\begin{equation}
\label{eq:metric}
g=dr^2 + \phi\left(\omega,r\right)^2g_{\mathbb{S}^2}, \quad \omega\in \mathbb{S}^2,
\end{equation}
where $g_{\mathbb{S}^2}$ is the standard metric on the unit sphere, and such that 
outside a compact set the curvature satisfies that, for $A>1$, it is bounded above
by $-A/r^2\log r$, the Dirichlet problem is solvable at infinity with respect to the pole.
Notice also, that due to the Uniformization Theorem, (\ref{eq:metric}) is a very general
form of writing a metric on $\mathbb{R}^3$.

\medskip
Our proof of Theorems \ref{thm:main2} and \ref{thm:main3} are based 
on Theorem \ref{thm:existence_dirichlet}, which does not use any curvature assumption.
Indeed, to make
an effective use of 
Theorem \ref{thm:existence_dirichlet},
the curvature assumptions are employed to
obtain the required 
properties on $\phi$ so
that the aforementioned theorem applies,
and thus our main results can be extended to other cases. 
For instance, if the curvature satisfies $-A/r^2\log r$, $A>1$, on a neighborhood
of each expansive end, Theorems \ref{thm:main2} and Theorem \ref{thm:main3} are 
still valid. Furthermore, 
at the end of Section \ref{sect:Dirichlet} we show an example of a metric on
the neighborhood of an end whose curvature does not have a definite sign, no 
matter how far we go along the end, but
for which the Dirichlet problem at infinity (with respect to the end) is solvable
(Section \ref{subsect:an_example}).


\medskip
The Dirichlet problem at infinity on negatively curved manifolds has a rich history, starting
with Choi's seminal work \cite{Choi}. Among the most important results related to this problem, we find
the existence results of Anderson \cite{A} and Sullivan \cite{Sullivan}, who showed that if $M$ is a Cartan-Hadamard
manifold whose curvature is bounded above and below by two negative constants, then the
Dirichlet problem at infinity can be solved for any continuous boundary data. 
Since then,
one of the main focus has been understanding the role of the bound from below 
for the existence of solutions to the Dirichlet problem (\cite{Borbely, Ji}), as Ancona \cite{Ancona} has shown that, in
general, if there is no lower bound on the curvature there are three-manifolds
with sectional curvature bounded above by $-1$ such that the Dirichlet problem at infinity has no solution.

\subsection{The two-dimensional case}
The solvability 
of the Dirichlet Problem at Infinity in 
the significant case of Cartan-Hadamard surfaces (and thus with just one end) is well understood
due, in part, to
the work of R. Neel, who, using
Probabilistic techniques, gave a sharp condition on the function $\phi$ in (\ref{eq:metric}) that
ensures the solvability of the Dirichlet problem at infinity (\cite{Neel14, Neel21}). Our
results (and 
mainly Therorem \ref{thm:existence_dirichlet}), proved
usign more analytics methods, are a natural extension
of his work to the next lowest 
dimension and to more
complicated topologies. However, the ideas employed
in the three-dimensional case work quite well in dimension two,
an thus, regarding
the two-dimensional case, we 
shall prove the following.
\begin{theorem}
\label{thm:twodim}
    Let $\left(M^2,g\right)$ be a surface with a pole. Then, 
    the metric $g$ can be written as
    \[
    g=dr^2+\phi\left(\theta,r\right)^2d\theta^2.
    \]
    If there is a positive 
    $\overline{\phi}\left(r\right)$
    such that for all $r>0$ large
    enough and all $\theta$ satisfies
    \[
    0\leq
\frac{\overline{\phi}_r\left(r\right)}{\overline{\phi}\left(r\right)}
    \leq \frac{\phi_r\left(\theta, r\right)}{\phi\left(\theta,r\right)},
     \]
     and such that
     \[
     \int^{\infty}\frac{1}{
     \overline{\phi}\left(s\right)}\,
     ds<\infty,
     \]
     then the Dirichlet problem
     at infinity is solvable for
     any continuous boundary data.
\end{theorem}
If the metric in $\mathbb{R}^2$
is rotationally symmetric, then
$\overline{\phi}$ can be taken
as $\phi$ itself, and hence the
previous theorem generalises
Milnor's result in \cite{Milnor77}
(at least to decide whether the
surface is hyperbolic), and also
the main result in \cite{Co2} in the two-dimensional case. Notice that Theorem
\ref{thm:twodim} does not require the 
surface to be Cartan-Hadamard. Using Milnor's argument in \cite{Milnor77},
it can be shown that if the surface is Cartan-Hadamard and the
sectional curvature of the 
surface is less than
$-A/r^2\log r$ for $A>1$, then
the Dirichlet Problem 
at Infinity is solvable. This
observation, together with the arguments
presented in Section 3.1 of
\cite{Co} give an alternative
new proof of Milnor's criteria
for surfaces with a pole. The
previous theorem can also be
extended to surfaces
of multiple ends, and there is
an analogue statement of the
in the three-dimensional
case, but we shall leave the
task to the interested reader.

\subsection{} As indicated in the previous paragraphs, 
most results on the Dirichlet problem at infinity are for \emph{simply connected manifolds}
with just one single exception known to the authors: the results of Choi \cite{Choi} for surfaces with 
\emph{finitely} many ends. Thus, the
results in this paper represent an advance in understanding the 
Dirichlet problem on manifolds with multiple ends.
Also, although it is not directly
implied by their statements, in the course of the proof of our results the reader will notice that 
\emph{as soon as a three-manifold has an expansive end with a neighborhood 
(of the end) whose curvature is bounded from above by a negative constant, it also has a wealth of 
bounded nonconstant harmonic functions (and thus,
Liouville's theorem does not hold).} This, as is well known, has implications on the behavior of the Brownian motion
on the manifold \cite{Grigoryan}.

\medskip
The idea of writing this paper
comes from a letter from professor Neel who thought that the ideas
presented in \cite{Co2} could be used to study the Dirichlet problem at
infinity for manifolds with multiple ends: the authors want to thank him for sharing his insight.
However, in the case of \cite{Co2}, a more direct method using Fourier expansions is
employed, whereas in this paper, as a direct attack with the methods of \cite{Co2} did not work, 
we had to combine them with Perron's method in order to prove our main results, being
the main difficulty with this approach to produce adecuate barriers at the points at 
infinity.
We must observe that whereas Perron's method is outstanding in its simplicity, the methods proposed in
\cite{Co2} allowed a bit more general boundary data, 
as we were able to show solvability in the case of boundary data in $L^2$.

\medskip
This paper is organised as follows. In Section \ref{section:preliminaries} we give 
and prove some preliminary lemmas. In Section \ref{sect:Dirichlet}, we prove
the solvability of the Dirichlet problem at infinity with respect to an end, and then
apply this result to give a proof of our main results in \ref{subsect:proofmain} and 
\ref{subsect:2D}.

\medskip

\section{Preliminaries}
\label{section:preliminaries}
In this section, we collect some
results that will be needed 
in the proofs of Theorems \ref{thm:main2} and \ref{thm:main3}. 





\begin{lemma}
\label{lemma:curvature}
Assume $\left(N, g_N\right)$ is locally conformally flat. Consider the metric
\[
g=dr^2+\phi^2\left(\omega,r\right) g_N,
\]
on $N\times\left[0,\infty\right)$.
Then for any direction $e_\sigma$ tangent to $N$, we have that 
the sectional curvature satisfies
\[
R_{r\sigma\sigma r}=-\frac{\phi_{rr}}{\phi}.
\]
\end{lemma}
\begin{proof}
As $g_N$ is locally conformally flat, it can be 
written, locally, as $g_N=\psi\left(\omega\right)^2\delta_{\alpha\beta}$. In the
following calculation, we use the expression of the curvature tensor
in local coordinates.
\[
R_{r\alpha \alpha}^{\quad\,\,\, r}=-\partial_{\alpha}\Gamma^{r}_{\alpha r}+\partial_r\Gamma^{r}_{\alpha\alpha}
-\Gamma_{\alpha j}^{r}\Gamma_{r\alpha}^j+\Gamma_{rj}^{r}\Gamma^j_{\alpha\alpha}.
\]
It is straightforward to compute that
\[
\Gamma^r_{rj}=0, \quad \mbox{for all} \quad j, \quad \Gamma_{\alpha\beta}^r=0,
\]
\[
\Gamma^{r}_{\alpha\alpha}=-\frac{1}{2}g_{\alpha\alpha,r}=-\phi\phi_r \psi^2
\]
and
\[
\Gamma_{r\alpha}^\alpha=\frac{1}{2}g^{\alpha\alpha}g_{\alpha\alpha,r}=\frac{\phi_r}{\phi},
\]
and hence,
\[
R_{r\alpha \alpha}^{\quad\,\,\, r}=-\phi\phi_{rr}\psi^2-\phi_r^2\psi^2+\phi_r^2\psi^2=
-\phi\phi_{rr}\psi^2,
\]
from which the claimed formula for the sectional curvature follows.

\end{proof}

The following basic tool is a classical comparison result due to Sturm.

\begin{lemma}[Comparison Lemma]
Let $u$ and $v$ be positive strictly increasing functions on $\left[a,\infty\right)$.
\begin{itemize}
    \item{(a)} If $u\left(a\right)<v\left(a\right)$, $u'\left(a\right)<v'\left(a\right)$ and
    $u''/u\leq v''/v$ on  $\left[a,\infty\right)$ then $u\leq v$ on $\left[a,\infty\right)$.
    \item{(b)}
     If $u'\left(a\right)/u\left(a\right)<v'\left(a\right)/v\left(a\right)$ and
    $u''/u\leq v''/v$ on  $\left[a,\infty\right)$ then $u'/u\leq v'/v$ on $\left[a,\infty\right)$.
\end{itemize}
\end{lemma}

\subsection{Perron's method}

A function $s:M\longrightarrow \mathbb{R}$ is called subharmonic if 
for any bounded open set $\Omega \subset M$ if for any $h$ a continuous harmonic
function such that $h=s$ on $\partial \Omega$ then $h\geq s$. A function $v$ is
called superharmonic if the opposite inequality holds. Observe that
when $s$ is $C^2$ (resp. $v$ is $C^2$), then it satisfies that $\Delta s\geq 0$
(resp. $\Delta v\leq 0$); and viceversa: when a $C^2$ satisfies that $\Delta s\geq 0$
(resp. $\Delta v\leq 0$) then it is subharmonic (resp. superharmonic).

\medskip
Given a set of boundary data, consider the set of continuous subharmonic functions 
$s:M\longrightarrow \mathbb{R}$ such that $s$ can be extended continuously
to the compactification of a neighborhood of each end, and 
such that for the compactification $N\times\left[0,\infty\right]$ we have that
$s|_{N_j\times \left\{\infty\right\}}\leq f_j$. Then,
\emph{the Perron solution} to the Dirichlet problem is given by
\[
u=\sup_{s\in \mathcal{S}} s.
\]
Next we recall the concept of a barrier. A \emph{barrier} $b$
on the neighborhood $Z=N\times\left[0,\infty\right)$ of an end 
at a point $\left(p,\infty\right)$
is a nonnegative superharmonic function $b: \overline{Z}\longrightarrow \mathbb{R}$ such that $b\left(p,\infty\right)=0$ and 
so that
$b>0$ on $\overline{Z}\setminus \left\{\left(p,\infty\right)\right\}$
(here $\overline{Z}=N\times\left[0,\infty\right]$).
A point
$\left(p,\infty\right)$ for which a barrier can be constructed is called 
\emph{regular boundary point}.

\medskip
The concept of barrier may seem too strong as to be of use or
easy to construct. However, there
is the concept of \emph{local barrier}, and these local barriers 
are what we will actually construct to prove
our main results. A local barrier is
 a superharmonic function defined on a neighborhood $U\subset \overline{N}$ of
 $\left(p,\infty\right)$ which satisfies the definition given above
 on $U$ (see page 25 in \cite{Gilbarg}). It is well known that from local barriers, barriers
can be constructed as is described in \cite{Gilbarg}.

\medskip
Barriers are useful because if a barrier can be constructed at every boundary point, then
the Perron solution to the Dirichlet problem will satisfy the boundary conditions. A 
proof of this fact can be found in Lemma 2.13 in \cite{Gilbarg}.


\section{The Dirichlet problem in three-dimensional manifolds}
\label{sect:Dirichlet}

\subsection{The Dirichlet problem on a neighborhood of an end}
In this section we first study metrics $g$ on $Z=N\times \left[0,\infty\right)$,
where $N$ is a closed surface, and
the metric is of the form
\[
g=dr^2+\phi\left(\omega,r\right)^2g_N.
\]
To this end, we will first compute the Laplacian of this metric. We have the following lemma.

\begin{lemma}
\label{lemma:Laplacian}
Assume $\mbox{dim}\left(N\right)=n-1$, $n\geq 3$, then
\[
\Delta_g = \frac{\partial^2}{\partial r^2}+\left(n-1\right)\frac{\phi_r}{\phi}\frac{\partial}{\partial r}
+\frac{1}{\phi^2}\Delta_N-\left(n-3\right)\frac{1}{\phi^3}\nabla_N \phi\cdot \nabla_N,
\]
where the subindex $N$ indicates that operators are defined with respect to the metric $g_N$.
\end{lemma}
\begin{proof}
The metric $g$ can be written in matrix notation as follows
\[
g_{ij}=
\left(
\begin{array}{cc}
1 & 0\\
0 & \phi^2 g_N
\end{array}
\right)
\]
and hence
\[
g^{ij}=
\left(
\begin{array}{cc}
1 & 0\\
0 & \dfrac{1}{\phi^2} g^{-1}_N
\end{array}
\right).
\]
This gives the following identity for the determinant of the metric:
\[
\left|g\right|=\phi^{2\left(n-1\right)}\left|g_N\right|.
\]
Therefore,
\begin{eqnarray*}
\Delta &=& \frac{1}{\phi^{\left(n-1\right)}\sqrt{\left|g_N\right|}}
\left[\frac{\partial}{\partial r}\left(\sqrt{\left|g\right|}\frac{\partial}{\partial r}\right)
+\frac{\partial}{\partial \omega^i}\left(\sqrt{\left|g\right|}
\frac{1}{\phi^2}g_N^{ij}\frac{\partial}{\partial \omega^j}\right)\right]\\
&=&\frac{\partial^2}{\partial r^2}+\left(n-1\right)\frac{\phi_r}{\phi}\frac{\partial}{\partial r}\\
&&+  \frac{1}{\phi^{\left(n-1\right)}\sqrt{\left|g_N\right|}}
\frac{\partial}{\partial \omega^i}\left(\phi^{n-3}\sqrt{\left|g_N\right|}g_N^{ij}\frac{\partial}{\partial \omega^j}\right)\\
&=&\frac{\partial^2}{\partial r^2}+\left(n-1\right)\frac{\phi_r}{\phi}\frac{\partial}{\partial r}\\
&&+  \frac{1}{\phi^2}\Delta_N +
\left(n-3\right)\frac{1}{\phi^3}\frac{\partial \phi}{\partial \omega^i}g_N^{ij}
\frac{\partial}{\partial \omega^j},
\end{eqnarray*}
which gives the result.
\end{proof}
Thus, when $n=3$ we have the following expression for the Laplacian, which is to be used below:
\begin{equation}
\label{eq:laplacian3d}
\Delta_g = \frac{\partial^2}{\partial r^2}+2\frac{\phi_r}{\phi}\frac{\partial}{\partial r}
+\frac{1}{\phi^2}\Delta_N.
\end{equation}


\medskip
We will first prove that local barriers can be constructed on an end, under some technical 
assumptions.
\begin{theorem}
\label{thm:existence_dirichlet}
Let $Z=N\times\left[0,\infty\right)$, $N$ a compact Riemannian surface, endowed with a metric of the form
\[
g=dr^2+\phi\left(\omega,r\right)^2g_N.
\]
Assume that there is a function $\overline{\phi}>0$  that only depends on $r$ and such that
there is $r_0>0$ such that $\overline{\phi}\left(r\right)\leq \phi\left(\omega,r\right)$ (for
all $\omega\in N$), and
\[
0<\frac{\overline{\phi}_r}{\overline{\phi}}\left(r\right)\leq \frac{\phi_r}{\phi}\left(\omega, r\right),
\]
for all $r\geq r_0$ (and all $\omega \in N$), and also that
\[
\int_{0}^{\infty}\dfrac{1}{\overline{\phi}}\,dr<\infty.
\]
Then, every point $\left(z,\infty\right)\in N\times\left\{\infty\right\}$ is a regular boundary point.
\end{theorem}

\begin{proof}
Let us construct a local barrier at any given point $\left(p,\infty\right)$, $p\in N$,
on $Z$. 
We let $U$ be a small open neighborhood of $p$ in $N$. We may assume that
the metric on $U\times\left[0,\infty\right)$ can be written as 
\[
dr^2+\mu\left(\omega,r \right)^2g_E,
\]
where $g_E$ is any flat metric on $U$, as 
every metric on a surface is locally conformally flat. Indeed, assume that $\phi$ satisfies 
the hypothesis of the theorem; we shall check that we can construct 
$\overline{\mu}$ such that $\mu$ satisfies the same hypothesis
in $U\times \left[r_0, \infty\right)$ as $\phi$ does, that is, that we also have
\[
\overline{\mu}\left(r\right)\leq \mu\left(\omega,r\right),
\]
\[
0<\frac{\overline{\mu}_r}{\overline{\mu}}\left(r\right)\leq \frac{\mu_r}{\mu}\left(\omega, r\right),
\]
and
\[
\int_{0}^{\infty}\dfrac{1}{\overline{\mu}}\,dr<\infty,
\]
on $U\times\left[r_0,\infty\right)$.
Indeed, observe that
we must have that $\mu\left(\omega,r\right)=\psi\left(\omega\right)\phi\left(\omega,r\right)$ on $U$,
for a convenient smooth function $\psi$.
By shrinking $U$ if necessary, we may assume that $\inf_{\omega\in U}\psi\left(\omega\right)=\eta>0$.
Thus,
\[
\overline{\phi}\left(r\right)\leq \frac{\psi\left(\omega\right)}{\psi\left(\omega\right)}\phi\left(\omega,r\right)
\leq \frac{1}{\eta}\mu\left(\omega, r\right). 
\]
Thus, take $\overline{\mu}=\eta\overline{\phi}$. It is not difficult to check that
with this choice of $\overline{\mu}$ the hypothesis of the theorem are satisfied
by $\mu$ as claimed. Since such a $\mu$ can be constructed from the function $\phi$,
we shall assume that $\mu=\phi$ in what follows.

\medskip
To proceed with the
construction of the local barrier,
without loss of generality, we assume that $U$ is a ball centered at $p$ with respect to 
the metric $g_E$.
Now, let $\vartheta$ be a first eigenfunction for the Dirichlet problem on $U$,
that is, 
\[
\Delta \vartheta=-\lambda_1^2 \vartheta \quad \mbox{in}\quad U, \quad \vartheta=0 \quad \mbox{on}
\quad \partial U,
\]
with $\lambda_1>0$.
Without loss of generality, we may assume that $\vartheta<0$ in $U$. It is 
well known that $\vartheta$ has an absolute minimum at $p$. We now define
\[
\sigma\left(r\right) = e^{-\int_r^{\infty} \frac{\lambda_1}{\overline{\phi}\left(s\right)}\,ds},
\quad r\geq r_0,
\]
with $\overline{\phi}$ as in the hypotheses in the statement of the theorem,
and consider the function $\Theta = \sigma\left(r\right) \vartheta\left(\omega\right)$.
We first show that $\Theta$ is superharmonic
on $U\times \left[r_0,\infty\right)$. Indeed, we begin by estimating as follows
\begin{eqnarray*}
\frac{d^2 \sigma}{dr^2}+2\frac{\phi_r}{\phi}\frac{d\sigma}{dr}-
\frac{\lambda_1^2}{\phi^2}\sigma & =& 
\left(\frac{\lambda_1^2}{\overline{\phi}^2}-\frac{\lambda_1\overline{\phi}_r}{\overline{\phi}^2}\right)\sigma
+2\lambda_1\frac{\phi_r}{\phi}\frac{1}{\overline{\phi}}\sigma
-
\frac{\lambda_1^2}{\phi^2}\sigma\\
&\geq& 
\frac{\lambda_1^2}{\overline{\phi}^2}\sigma
+\lambda_1\frac{\phi_r}{\phi}\frac{1}{\overline{\phi}}\sigma
-
\frac{\lambda_1^2}{\phi^2}\sigma\\
&\geq& 0,
\end{eqnarray*}
where we have used that $\overline{\phi}\leq \phi$ and that $0<\overline{\phi}_r/\overline{\phi}\leq \phi_r/\phi$ 
for $r\geq r_0$. Therefore,
\[
\frac{d^2 \sigma}{dr^2}+2\frac{\phi_r}{\phi}\frac{d\sigma}{dr}\geq
\frac{\lambda_1^2}{\phi^2}\sigma=-\frac{1}{\phi^2}\frac{\sigma\Delta_N \vartheta}{\vartheta},
\]
and since $\vartheta<0$, we obtain 
\[
\vartheta\frac{d^2 \sigma}{dr^2}+2\frac{\phi_r}{\phi}\vartheta\frac{d\sigma}{dr}+
\frac{1}{\phi^2}\sigma\Delta_N \vartheta\leq 0.
\]
Now, using (\ref{eq:laplacian3d}) yields
\[
\Delta \Theta = \vartheta\frac{d^2 \sigma}{dr^2}+2\frac{\phi_r}{\phi}\vartheta\frac{d\sigma}{dr}+
\frac{1}{\phi^2}\sigma\Delta_N \vartheta,
\]
and so we get that
\[
\Delta \Theta \leq 0.
\]
What we have thus far is that $\Theta$ is almost a local barrier at $\left(p,\infty\right)$, but we still
need to modify it a little bit (besides being superharmonic, we also need a local barrier to be nonnegative,
and its value at $\left(p,\infty\right)$ to be $0$). In order to do so,
let $A=\vartheta\left(p\right)$
(which, recall, besides being strictly negative is an absolute minimum),
and define
\[
\Theta_A= \Theta - A.
\]
Clearly, $\Theta_A$ is nonnegative, and it is actually positive 
on $\overline{U}\times\left[r_0,\infty\right]$,
except at $\left(p,\infty\right)$ where
its value is 0. Thus $\Theta_A$ is a local barrier at $\left(p,\infty\right)$. 
\end{proof}

From the previous theorem, 
we can deduce the regularity of any point $\left(p,\infty\right)$ on the neighborhood of an expansive end when the sectional curvature of a neighborhood satisfies
that the sectional curvature is $K_Z\leq -a^2<0$. Indeed, in this case we must have
by Lemma \ref{lemma:curvature} that
\[
\frac{\phi_{rr}}{\phi}\geq a^2.
\]
Thus, let $\overline{\phi}=\alpha \sinh\left[ar+1\right]$. Choose $\alpha>0$ 
 so that 
\[
\alpha \sinh 1< \min_{\omega}\phi\left(\omega, 0\right),
\]
\[
a\alpha < \min_{\omega\in N}\frac{\phi_r\left(\omega,0\right)}{\phi\left(\omega,0\right)},
\]
and
\[
\alpha a \cosh 1 < \min_{\omega}\phi_r\left(\omega, 0\right).
\]
Since $\overline{\phi}_{rr}/\overline{\phi}=a^2$,
by the Comparison Lemma, the hypotheses of Theorem \ref{thm:existence_dirichlet} hold,
and thus we get:
\begin{theorem}
\label{thm:existence_dirichlet2}
Let $Z=N\times\left[0,\infty\right)$, $N$ a compact Riemannian surface, endowed with a metric of the form
\[
g=dr^2+\phi\left(\omega,r\right)^2g_N.
\]
Assume that the sectional curvature of $Z$ satisfies that $K_{Z}\leq -a^2$, $a>0$, and that the end is expansive.
Then, any point on $N\times\left\{\infty\right\}$ is a regular boundary point.
\end{theorem}
\subsubsection{An example} 
\label{subsect:an_example}
Assume that on the neighborhood of an end we have that for $r>0$ large enough
the metric is given by
\[
dr^2+\phi\left(r\right)^2g_N,
\]
with
\[
\phi\left(r\right)=\sin r+r\log^{2}r.
\]
Then, in this case,  
it is a simple matter to compute the curvature in the direction of $r$:
\[
-\frac{\phi''}{\phi}=-\frac{-\sin r+\frac{2\log r}{r}+\frac{2}{r}}{\sin r+r\log^2 r},
\]
which can be positive or negative for large $r$. Notice that 
$\phi'\geq 0$ for large $r$ so that the end is expansive. In this case,
we can take as $\phi$ as  $\overline{\phi}$, and it is easy to show that
\[
\int^{\infty} \frac{1}{\phi\left(s\right)}\,ds <\infty,
\]
and thus the Dirichlet problem at infinity is solvable with respect to this end.

\subsection{Three-manifolds with multiple ends: A proof of Theorems \ref{thm:main2} and \ref{thm:main3}}
\label{subsect:proofmain}
Theorem \ref{thm:existence_dirichlet2} immediatly implies 
Theorems \ref{thm:main2} and \ref{thm:main3}. Indeed,
the set of subharmonic functions such that their boundary values
are less than the prescribed boundary values is nonempty. In the case of
infinitely many ends, here is where we need the boundary data to be bounded from below.
Further, since all of these subharmonic functions are bounded above by
a bound on the boundary data (by the Maximum Principle), their supremum does exist. 
Since we are assuming expanding ends with neighborhoods
with sectional curvatures bounded from above by a negative constant, all boundary points of each end are
regular, and thus the Dirichlet problem is solvable (Perron's method).

\medskip
Finally, we show that whenever a manifold with
cylindrical ends has an expansive end such that 
it has a neighborhood which satisfies that its curvature is bounded above by a 
negative constant, then it has plenty of bounded harmonic functions.
To proceed, let $Z=N\times\left[0,\infty\right)$ be a neighborhood of the expansive 
end, and assume that $K_Z\leq -a^2<0$: pick $f\in C\left(N\right)$ continuous
and nonconstant, and let $\mathcal{S}$ be the set of continuous
subharmonic functions defined on $M$, such that they can be extended
continuously to the compactification of each of the neighborhoods 
of the ends, and such that at the boundary of the 
neighborhood of each end, different 
from $Z$, they satisfy that they are less or equal than $\min_{p\in N} f\left(p\right)$,
and such that at 
$N\times \left\{\infty\right\}$ they are pointwise less than or equal than $f$. Then
$\mathcal{S}$ is nonempty; let $u=\sup_{s\in\mathcal{S}} s$ be the Perron
solution to the problem. It is clear that
$u$ is harmonic and bounded, and since $Z$ is expanding, then
every point on $N\times\left\{\infty\right\}$ is regular, and thus
$u$ satisfies the boundary condition at $N\times \left\{\infty\right\}$,
and hence $u$ is nonconstant as claimed.

\subsection{The two dimensional
case}
\label{subsect:2D}

Following the arguments given in the three-dimensional case, all we have to do is
to construct a local barrier at any point $\left(\theta_0,\infty\right)$ in the
circle at infinity. 
Using that the Laplacian in this case can be written as
\[
\Delta=\frac{\partial^2}{\partial r^2}+\frac{\phi_r}{\phi}\frac{\partial}{\partial r}
+\frac{1}{\phi^2}\frac{\partial^2}{\partial\theta^2},
\]
using the assumptions of Theorem \ref{thm:twodim}, we can explicitly show that for any $a>0$,
the function
\[
u=-\left[\exp\left(-\int_r^{\infty}\frac{1}{\overline{\phi}\left(s\right)}\right)\,ds\right] 
\cos\left(\theta-\theta_0\right)
\]
is superharmonic in $\left(\theta,r\right)\in\left(\theta_0-\frac{\pi}{2},\theta_0+\frac{\pi}{2}\right)\times \left(R,\infty\right)$
for $R>0$ large enough. 
Indeed,
\[
\Delta u = \frac{1}{\overline{\phi}}\left(\frac{\phi_r}{\phi}-\frac{\overline{\phi}_r}{\overline{\phi}}\right)u
+\left(\frac{1}{\overline{\phi}^2}-\frac{1}{\phi^2}\right)u\leq 0,
\]
whenever $\cos\left(\theta-\theta_0\right)'\geq 0$.

\medskip
Now, let $A=1 \left(=\lim_{r\rightarrow\infty}u\left(r,\theta_0\right)\right)$, and define
\[
v= u-A.
\]
In the domain $\left(-\frac{\pi}{2},\frac{\pi}{2}\right)\times\left(R,\infty\right)$
is a local barrier at $\left(\theta_0,\infty\right)$.

\medskip
Notice that Theorem \ref{thm:twodim} allows metrics on $\mathbb{R}^2$ that are not Cartan-Hadamard.
For instance, if outside a compact set we have that
$g=dr^2+\phi\left(r\right)^2d\theta^2$,
$\phi\left(r\right)=\sin r+r\log^2 r$, then
the Dirichlet problem is solvable at infinity.

\medskip
Finally,  as said in the Introduction, Theorem \ref{thm:twodim} and Theorem 3 and
the arguments presented in Section 3 of \cite{Co} 
(where a version of Hadamard's three-circle theorem is presented), give 
a new proof of Milnor's criteria for parabolicity of a surface with
a pole.




\end{document}